\documentclass[10pt,reqno]{amsart}
\usepackage{amssymb,amscd}   

\newtheorem{theorem}{Theorem}[section]
\newtheorem{proposition}[theorem]{Proposition} 
\newtheorem{corollary}[theorem]{Corollary}
\newtheorem{lemma}[theorem]{Lemma}

\numberwithin{equation}{section}

\begin{document}
\title[Type III$_1$ Factors Arising As Free Products]{On Type III$_1$ Factors Arising as Free Products}
\author[Y. Ueda]
{Yoshimichi UEDA}
\address{
Graduate School of Mathematics \\
Kyushu University \\ 
Fukuoka, 819-0395, Japan
}
\email{ueda@math.kyushu-u.ac.jp}
\thanks{Supported by Grant-in-Aid for Scientific Research (C) 20540213.}

\maketitle

\begin{abstract}Type III$_1$ factors arising as (direct summands of) von Neumann algebraic free products are investigated. In particular we compute Connes' Sd- and $\tau$- invariants for those type III$_1$ factors without any extra assumption.
\end{abstract}

\allowdisplaybreaks{

\section{Introduction} Let $M_1$ and $M_2$ be $\sigma$-finite von Neumann algebras equipped with faithful normal states $\varphi_1$ and $\varphi_2$, respectively. The von Neumann algebraic free product $(M,\varphi)$ of $(M_1,\varphi_1)$ and $(M_2,\varphi_2)$ has been  investigated so far by many hands. The questions of factoriality, Murray--von Neumann--Connes type classification and fullness for $M$ are fundamental, but it is quite recent that they are solved completely. The complete solution \cite{Ueda:Preprint10} is as follows. When $M_1 \neq \mathbb{C} \neq M_2$, the resulting free product von Neumann algebra $M$ is always of the form $M_d\oplus M_c$ of multi-matrix algebra $M_d$ and diffuse von Neumann algebra $M_c$ possibly with $M=M_c$ or equivalently $M_d=0$. The center of $M$ and the structure of $(M_d, \varphi|_{M_d})$ are described explicitly. If $\mathrm{dim}(M_1)=2=\mathrm{dim}(M_2)$, then $M_c \cong L^\infty[0,1]\otimes M_2(\mathbb{C})$ and the free product state $\varphi$ is tracial; otherwise $M_c$ is always a factor of type II$_1$ or III$_\lambda$ ($\lambda\neq 0$) with the T-set formula $T(M_c) = \{t \in \mathbb{R}\,|\,\sigma_t^{\varphi_1} = \mathrm{Id} = \sigma_t^{\varphi_2}\}$ and $M_c'\cap M_c^\omega = \mathbb{C}$. 

Many fundamental questions on $M_c$ still remain unsolved, and some of them were stated in \cite[\S5]{Ueda:Preprint10}. The main purpose of the present notes is to study some of those, and thus this means that the present work may be regarded as a continuation of \cite{Ueda:Preprint10}. The most interesting question is roughly how the free product von Neumann algebra $M$ `remembers' the given states $\varphi_i$, $i=1,2$. In view of this we are interested in when $((M_c)_{\varphi|_{M_c}})'\cap M_c^\omega$ becomes trivial. It was already established, see \cite[Remark 4.2 (4)]{Ueda:Preprint10}, that this is the case at least when both $M_1$ and $M_2$ are of type I with discrete center. In the present notes we show that this is always the case if the given $\varphi_1$ and $\varphi_2$ are almost periodic. This in particular shows that the Sd-invariant $\mathrm{Sd}(M_c)$ is exactly the multiplicative group algebraically generated by the point spectra of the modular operators $\Delta_{\varphi_i}$, $i=1,2$, under the same hypothesis on $\varphi_1$ and $\varphi_2$ as above together with the separability of preduals. We also give an opposite fact, that is, any finite von Neumann algebra (allowed to be any multi-matrix algebra) can be the centralizer of the free product state on a certain free product type III$_1$ factor (which depends on a given finite von Neumann algebra). Remark that the modular operator associated with the free product state that we construct has no eigenvalue except $1$. Therefore the next task should be to clarify when $M_c$ has an almost periodic state. In fact, there may a priori exist an almost periodic state on $M_c$ that is different from the free product state $\varphi|_{M_c}$. For the question we compute the $\tau$-invariant $\tau(M_c)$ introduced by Connes \cite{Connes:JFA74} in terms of given data {\it without any extra assumption} (except the separability of preduals). This is nothing but a final result in the direction, generalizing all the previously known results due to Shlyakhtenko \cite[Corollary 8.4]{Shlyakhtenko:JFA} (based on Barnett's work \cite[Theorem 11]{Barnett:PAMS95}) and Vaes \cite[Appendix 4]{Vaes:Bourbaki05}. Our result on $\tau(M_c)$ implies, under the separability assumption of preduals again, that $M_c$ possesses an almost periodic state or weight if and only if the given $\varphi_1$ and $\varphi_2$ are almost periodic, or equivalently so is $\varphi$.   

The notations we employ in the present notes entirely follow our previous paper \cite{Ueda:Preprint10}. Basics on von Neumann algebraic free products and ultraproducts of von Neumann algebras are summarized in \cite[\S2]{Ueda:Preprint10}. We refer to \cite{Takesaki:Book2} for standard theory on von Neumann algebras including modular theory, and to Connes' paper \cite{Connes:JFA74} (and also Shlyakhtenko \cite{Shlyakhtenko:TAMS05}) for general scheme of analysis on full factors of type III$_1$.  

\section{Relative Commutant $((M_c)_{\varphi|_{M_c}})'\cap M_c^\omega$ and Sd-invariant $\mathrm{Sd}(M_c)$} 

Firstly we establish that any free product state of almost periodic ones satisfies a very strong ergodicity property. We begin with the next lemma. It should be a folklore, but we give a proof for the sake of completeness.  

\begin{lemma}\label{L-2.1} Let $N$ be a $\sigma$-finite von Neumann algebra equipped with an almost periodic state $\psi$. If $N$ is diffuse, then the centralizer $N_\psi$ must be diffuse. 
\end{lemma} 
\begin{proof} On contrary, suppose that there is a minimal nonzero $e \in N_\psi^p$. By the characterization of modular automorphisms \cite[Theorem VIII.1.2]{Takesaki:Book2} we have $\sigma_t^{\psi|_{eNe}} = \sigma^\psi|_{eNe}$, $t \in \mathbb{R}$, so that $(eNe)_{\psi|_{eNe}} = eN_\psi e = \mathbb{C}e$. By  \cite[Lemma 3.7.5.(c)]{Connes:ASENS73} $\psi|_{eNe}$ is almost periodic again, and thus $\Delta_{\psi|_{eNe}} = \sum_{\lambda>0} \lambda\,E_{\Delta_{\psi|_{eNe}}}(\{\lambda\})$. Since $N$ is diffuse, $eNe \supsetneqq \mathbb{C}e  = (eNe)_{\psi|_{eNe}}$ and thus $E_{\Delta_{\psi|_{eNe}}}(\{\lambda_0\}) \neq 0$ for some $\lambda_0 \neq 1$ (otherwise $\Delta_{\psi|_{eNe}}$ must be $1$; implying $eNe = (eNe)_{\psi|_{eNe}} = \mathbb{C}e$, a contradiction). By \cite[Lemma 1.12]{Takesaki:ActaMath73} there is a non-zero partial isometry $v \in eNe$ so that $\sigma_t^{\psi|_{eNe}}(v) = \lambda_0^{it}v$. Since $v^* v, vv^* \in (eNe)_{\psi|_{eNe}} = \mathbb{C}e$, the partial isometry $v$ is actually a unitary in $eNe$. However, by e.g.~\cite[Lemma 1.6]{Takesaki:ActaMath73} one has $\psi(e) = \psi|_{eNe}(vv^*) = \lambda_0\psi|_{eNe}(v^* v) = \lambda_0\psi(e)$, which is impossible due to $\psi(e) \neq 0$ and $\lambda_0 \neq 1$. Hence we are done. 
\end{proof} 

The above proposition with the proof of \cite[Theorem 4.1]{Ueda:Preprint10} enables us to show the next theorem.  

\begin{theorem}\label{T-2.2} Let $M_1$ and $M_2$ be non-trivial $\sigma$-finite von Neumann algebras equipped with faithful normal states $\varphi_1$ and $\varphi_2$, respectively. Denote by $(M,\varphi)$ their free product, and assume that $\mathrm{dim}(M_1) + \mathrm{dim}(M_2) \geq 5$. Then, if both $\varphi_1$ and $\varphi_2$ are almost periodic, then the restriction of $\varphi$ to the diffuse factor part $M_c$ {\rm(}see \S1{\rm)} must be almost periodic and satisfy the very strong ergodicity $((M_c)_{\varphi|_{M_c}})'\cap M_c^\omega = \mathbb{C}$. 
\end{theorem}
\begin{proof} The first assertion follows from \cite[Proposition 4.2]{Dykema:FieldsInstituteComm97} (together with \cite[Lemma 3.7.5.(c)]{Connes:ASENS73} when $M\neq M_c$). In fact, the explicit description of the modular operator $\Delta_\varphi$ associated with $\varphi$ (which is explained in \cite[\S1]{Dykema:Crelle94}) enables us to find a total subset consisting of its eigenvectors in the GNS Hilbert space associated with $\varphi$. Thus it suffices to show the latter assertion, i.e., $((M_c)_{\varphi|_{M_c}})'\cap M_c^\omega = \mathbb{C}$. 

Decompose $M_i = M_{id}\oplus M_{ic}$ into the `type I with discrete center' part and the diffuse part, $i=1,2$. We already confirmed, in \cite[Remark 4.2 (4)]{Ueda:Preprint10}, that $((M_c)_{\varphi|_{M_c}})'\cap M_c^\omega = \mathbb{C}$ at least when both $M_i = M_{id}$, $i=1,2$, hold. In the case that either $M_1 = M_{1c}$ or $M_2 = M_{2c}$ the desired $(M_\varphi)'\cap M^\omega = \mathbb{C}$ holds by the latter assertion of \cite[Theorem 3.7]{Ueda:Preprint10} since one free component has the diffuse centralizer and the other does the non-trivial one by Lemma \ref{L-2.1}.  
Thus it suffices to consider the following three cases: 
\begin{itemize} 
\item[(i)] $M_1 = M_{1d}\oplus M_{1c}$ and $M_2 = M_{2d}$, 
\item[(ii)] $M_1 = M_{1d}$ and $M_2 = M_{2d}\oplus M_{2c}$, 
\item[(iii)] $M_1 = M_{1d}\oplus M_{1c}$ and $M_2 = M_{2d}\oplus M_{2c}$.  
\end{itemize} 

Firstly we deal with the cases (i),(ii), but it suffices to consider only (i) by symmetry. Set $N := (M_{1d}\oplus\mathbb{C}1_{M_{1c}})\vee M_2$ inside $M$, which is the free product von Neumann algebra of two type I von Neumann algebras with discrete center. Write $p := 1_{M_{1c}}$ for short. By Dykema's free etymology technique (see e.g.~\cite[Lemma 2.2]{Ueda:Preprint10}) we have 
$$ 
(pMp,(1/\varphi_1(p))\varphi|_{pMp}) 
=
(M_{1c},(1/\varphi_1(p))\varphi_1|_{M_{1c}})\star
(pNp,(1/\varphi_1(p))\varphi_{pNp})   
$$
with $c_M(p) = c_N(p)$. By \cite[Theorem 4.1 and Remark 4.2 (4)]{Ueda:Preprint10} (or the structure theorem for two freely independent projections, see \cite[Example 3.6.7]{VDN} and \cite[Theorem 1.1]{Dykema:DukeMathJ93}) $N = N_d\oplus N_c$, the `type I with discrete center' part and the diffuse part, so that $(N_c)_{\varphi|_{N_c}}$ is diffuse, and moreover $p= p_d \oplus p_c \in (N_d\oplus N_c)\cap N_{\varphi|_N}$ satisfies that $c_{(N_c)_{\varphi|_{N_c}}}(p_c) = 1_{N_c}$ and either $p_d = 0$, $p_d = 1_{N_d}$, or $p_d$ is minimal and central in $N_d$. 
In particular, $(pNp)_{\varphi|_{pNp}} \neq \mathbb{C}p$. By Lemma \ref{L-2.1} $(M_{1c})_{\varphi_1|_{M_{1c}}}$ is diffuse. Therefore the latter assertion of \cite[Theorem 3.7]{Ueda:Preprint10} shows that $(p((M_c)_{\varphi|_{M_c}})p)' \cap (p M_c^\omega p) = ((pMp)_{\varphi|_{pMp}})'\cap (pMp)^\omega = \mathbb{C}p$ ({\it n.b.}~$p \in M_c$ by \cite[Theorem 4.1]{Ueda:Preprint10}). Consequently, if $c_{(M_c)_{\varphi|_{M_c}}}(p) = 1_{M_c}$ was once confirmed, then $((M_c)_{\varphi|_{M_c}})'\cap M_c^\omega = \mathbb{C}$ would immediately follow (see the beginning of the proof of \cite[Theorem 4.1]{Ueda:Preprint10}). When $p_d = 0$, then \cite[Theorem 4.1]{Ueda:Preprint10} (the explicit description of multi-matrix part) shows that $M_d = N_d$ (i.e., $1_{N_c} = 1_{M_c}$) and $\mathcal{Z}(M_c) = \mathbb{C}1_{M_c}$, implying $c_{(M_c)_{\varphi|_{M_c}}}(p) \geq c_{(N_c)_{\varphi|_{N_c}}}(p) = 1_{N_c} = 1_{M_c} = c_{M_c}(p)$ so that $c_{(M_c)_{\varphi|_{M_c}}}(p) = 1_{M_c}$. For the other cases we firstly observe that $c_{N_{\varphi|_N}}(p) = c_{(N_d)_{\varphi|_{N_d}}}(p_d)\oplus c_{(N_c)_{\varphi|_{N_c}}}(p_c) = c_{(N_d)_{\varphi|_{N_d}}}(p_d)\oplus 1_{N_c}$. When $p_d = 1_{N_d}$, one has $c_{N_{\varphi|_N}}(p) = 1$, implying $c_{M_\varphi}(p) = 1$ ({\it n.b.} $M=M_c$ in this case). When $p_d$ is minimal and central, \cite[Theorem 4.1]{Ueda:Preprint10} says that $M_d = N_d(1-p_d)$, and $c_{M_c}(p) = p_d\oplus 1_{N_c} = c_{N_\varphi}(p) = c_{\mathbb{C}p_d\oplus(N_c)_{\varphi|_{N_c}}}(p)$ by the above. Since $\mathbb{C}p_d\oplus(N_c)_{\varphi|_{N_c}}$ sits in $(M_c)_{\varphi|_{M_c}}$, we get $c_{(M_c)_{\varphi|_{M_c}}}(p) = p_d \oplus 1_{N_c} = 1_{M_c}$. Hence we are done in the cases (i),(ii). 

The exactly same argument as in the cases (i),(ii) with replacing \cite[Theorem 4.1, Remark 4.2 (4)]{Ueda:Preprint10} by the case (ii) (or (i)) shows that $((M_c)_{\varphi|_{M_c}})'\cap M_c^\omega = \mathbb{C}$ in the case (iii) too. The details are left to the reader. 
\end{proof}  

If $M_1$ and $M_2$ are assumed to have separable preduals, then automorphism analysis in \cite[\S4]{Connes:JFA74} is available, and the above theorem together with the explicit description of the modular operator associated with $\varphi$ implies the next corollary.  

\begin{corollary}\label{C-2.3} Let $M_1$ and $M_2$ be non-trivial von Neumann algebras with separable preduals, equipped with faithful normal states $\varphi_1$ and $\varphi_2$, respectively. Denote by $(M,\varphi)$ their free product, and assume that $\mathrm{dim}(M_1) + \mathrm{dim}(M_2) \geq 5$. If both $\varphi_1$ and $\varphi_2$ are almost periodic, then the Sd-invariant $\mathrm{Sd}(M_c)$ of the diffuse factor part $M_c$ is exactly the multiplicative group algebraically generated by the point spectra of $\Delta_{\varphi_i}$, $i=1,2$. 
\end{corollary}
\begin{proof} We have known that $M_c$ is full and $(M_c)_{\varphi|_{M_c}}$ a factor. Hence $\mathrm{Sd}(M_c)$ is exactly the point spectrum of the modular operator $\Delta_{\varphi|_{M_c}}$ associated with $\varphi|_{M_c}$ thanks to \cite[Lemma 4.8]{Connes:JFA74}. It is plain to see, by the explicit description of $\Delta_\varphi$, that the point spectrum of $\Delta_\varphi$ is the multiplicative group algebraically generated by those of $\Delta_{\varphi_i}$, $i=1,2$, see \cite[Proposition 4.2]{Dykema:FieldsInstituteComm97}. Thus the desired assertion follows if $M=M_c$. For the general case (i.e., $M=M_d\oplus M_c$ with $M_d \neq 0$) we need an exact relationship between $\Delta_\varphi$ and $\Delta_{\varphi|_{M_c}}$. 

By the characterization of modular automorphisms \cite[Theorem VIII.1.2]{Takesaki:Book2} one has $\sigma_t^{\varphi|_{M_c}} = \sigma_t^\varphi|_{M_c}$, $t \in \mathbb{R}$. Let $(M\curvearrowright \mathcal{H}_\varphi,\Lambda_\varphi)$ be the GNS representation associated with $\varphi$. Set $\mathcal{H}_0 := \overline{\Lambda_\varphi(M_c)}$ and $\Lambda_0 := \Lambda_\varphi|_{M_c} : M_c \rightarrow \mathcal{H}_0$. It is easy to see that the representation $M_c \curvearrowright \mathcal{H}_0$ with $\Lambda_0 : M_c \rightarrow \mathcal{H}_0$ can be identified with the GNS representation associated with $\varphi|_{M_c}$ so that we write $\mathcal{H}_{\varphi|_{M_c}} := \mathcal{H}_0$ and $\Lambda_{\varphi|_{M_c}} := \Lambda_0$. Denote by $P$ the projection from $\mathcal{H}_\varphi$ onto $\mathcal{H}_{\varphi|_{M_c}}$. As in (b) $\Rightarrow$ (a) of the proof of \cite[Theorem 7.1]{Takesaki:LNP20} one has $(1-2P)\Delta_\varphi(1-2P) = \Delta_\varphi$ so that $\Delta_\varphi$ is affiliated with $\{P\}'$, i.e., $\Delta_\varphi \eta \{P\}'$ on $\mathcal{H}_\varphi$. In particular, the spectral projection $E_{\Delta_\varphi}(-)$ of $\Delta_\varphi$ and $\Delta_\varphi^{it}$ ($t \in \mathbb{R}$) commute with $P$, and thus the restrictions $E_{\Delta_\varphi}(-)|_{\mathcal{H}_{\varphi|_{M_c}}}$ and $\Delta_\varphi^{it}|_{\mathcal{H}_{\varphi|_{M_c}}}$ to $\mathcal{H}_{\varphi|_{M_c}}$ are well-defined. Moreover the restriction $\Delta_\varphi|_{\mathcal{H}_{\varphi|_{M_c}}}$ to $\mathcal{H}_{\varphi|_{M_c}}$ is well-defined in the following sense: $\mathrm{Domain}(\Delta_\varphi|_{\mathcal{H}_{\varphi|_{M_c}}}) = \mathrm{Domain}(\Delta_\varphi) \cap \mathcal{H}_{\varphi|_{M_c}} = P(\mathrm{Domain}(\Delta_\varphi))$ and $(\Delta_\varphi|_{\mathcal{H}_{\varphi|_{M_c}}})\xi = \Delta_\varphi\xi$, $\xi \in \mathrm{Domain}(\Delta_\varphi|_{\mathcal{H}_{\varphi|_{M_c}}})$. In fact, $\xi \in \mathrm{Domain}(\Delta_\varphi|_{\mathcal{H}_{\varphi|_{M_c}}})$ if and only if $\xi \in \mathcal{H}_{\varphi|_{M_c}}$ and $\sum_{\lambda>0} \lambda^2 \Vert E_{\Delta_\psi}(\{\lambda\})\xi\Vert_{\mathcal{H}_\varphi}^2 < +\infty$. Moreover $\Delta_\varphi \eta \{P\}'$ implies $P\Delta_\varphi \subseteq \Delta_\varphi P$, and 
\begin{align*} 
((\Delta_\varphi|_{M_c})\xi|\zeta)_{\varphi|_{M_c}} = (\Delta_\varphi\xi|\zeta)_\varphi &= \sum_{\lambda>0} \lambda\,(E_{\Delta_\varphi}(\{\lambda\})\xi|\zeta)_\varphi \\ 
&= \sum_{\lambda>0}\lambda\,((E_{\Delta_\varphi}(\{\lambda\})|_{\mathcal{H}_{\varphi_{M_c}}})\xi|\zeta)_{\varphi|_{M_c}}
\end{align*} 
for $\xi \in \mathrm{Domain}(\Delta_\varphi|_{\mathcal{H}_{\varphi|_{M_c}}})$, $\zeta \in \mathcal{H}_{\varphi|_{M_c}}$. Those show that $\Delta_\varphi|_{\mathcal{H}_{\varphi|_{M_c}}} = (\Delta_\varphi|_{\mathcal{H}_{\varphi|_{M_c}}})^*$ and the spectral projection of $\Delta_\varphi|_{\mathcal{H}_{\varphi|_{M_c}}}$ is given by the restriction $E_{\Delta_\varphi}(-)|_{\mathcal{H}_{\varphi_{M_c}}}$. Hence, for every $x \in M_c$ and every $t \in \mathbb{R}$ we have 
\begin{align*} 
(\Delta_\varphi|_{\mathcal{H}_{\varphi|_{M_c}}})^{it}\Lambda_{\varphi|_{M_c}}(x) 
&= \Delta_\varphi^{it}\Lambda_\varphi(x) 
= \Lambda_\varphi(\sigma_t^\varphi(x))  \\
&= \Lambda_{\varphi|_{M_c}}(\sigma_t^{\varphi|_{M_c}}(x)) = (\Delta_{\varphi|_{M_c}})^{it}\Lambda_{\varphi|_{M_c}}(x), 
\end{align*} 
implying $\Delta_{\varphi|_{M_c}} = \Delta_\varphi|_{\mathcal{H}_{\varphi|_{M_c}}}$  and $E_{\Delta_{\varphi|_{M_c}}}(-) = E_{\Delta_\varphi}(-)|_{\mathcal{H}_{\varphi|_{M_c}}}$ by the uniqueness part of Stone's theorem. 

From the above fact it immediately follows that the point spectrum of $\Delta_{\varphi|_{M_c}}$ is contained in that of $\Delta_\varphi$, and we want to prove that they are exactly same. We have known that the point spectrum of $\Delta_\varphi$ is the multiplicative group algebraically generated by those of $\Delta_{\varphi_i}$, $i=1,2$, and also that the point spectrum of $\Delta_{\varphi|_{M_c}}$ is $\mathrm{Sd}(M_c)$ being a multiplicative group. Thus it suffices to prove that any eigenvalue of $\Delta_{\varphi_i}$, $i=1,2$, becomes an eigenvalue of $\Delta_{\varphi|_{M_c}}$. Let us choose an eigenvalue $\lambda \neq 1$ of $\Delta_{\varphi_i}$, $i=1,2$. By \cite[Lemma 1.12]{Takesaki:ActaMath73} we can choose a corresponding eigenvector of the form $\Lambda_{\varphi_i}(x)$, $x \in M_i$. It is plain to see that $\sigma_t^{\varphi_i}(x) = \lambda^{it}x$, $x \in \mathbb{R}$, and by using e.g.~\cite[Lemma 1.6]{Takesaki:ActaMath73} together with the formula `$\sigma_t^\varphi = \sigma_t^{\varphi_1}\star\sigma_t^{\varphi_2}$', $t \in \mathbb{R}$, we see that $x$ is an analytic element with respect to $\sigma^\varphi$ (indeed $z \in \mathbb{C} \mapsto \lambda^z x \in M$ is its unique analytic extension), and thus
$\Delta_\varphi\Lambda_\varphi(x) = \Lambda_\varphi(\sigma_{-i}^\varphi(x)) = \lambda\Lambda_\varphi(x)$ (see around \cite[Lemma VIII.2.4]{Takesaki:Book2}); saying that $\Lambda_\varphi(x)$ is an eigenvector of $\Delta_\varphi$ corresponding to $\lambda$. By the explicit description of the multi-matrix part $M_d$ in \cite[Theorem 4.1]{Ueda:Preprint10} one can easily see that $x = x_d\oplus x_c \in M_d\oplus M_c$ satisfies $x_c \neq 0$. 
Then $P\Lambda_\varphi(x) = \Lambda_\varphi(x_c) = \Lambda_{\varphi|_{M_c}}(x_c)$ gives a non-zero vector in $\mathcal{H}_{\varphi|_{M_c}}$, and $\Delta_{\varphi|_{M_c}}\Lambda_{\varphi|_{M_c}}(x_c) = \Delta_\varphi P\Lambda_\varphi(x) = P\Delta_\varphi\Lambda_\varphi(x) = \lambda P\Lambda_\varphi(x) = \lambda\Lambda_{\varphi|_{M_c}}(x_c)$. Thus $\lambda$ is an eigenvalue of $\Delta_{\varphi|_{M_c}}$.     
\end{proof}  

Or more less related to the above we are interested in the question: `Which finite von Neumann algebra can be the centralizer of a faithful normal state of a factor ?' In the direction Herman and Takesaki \cite{HermanTakesaki:CMP70} constructed the first example of the trivial centralizer. Connes \cite{Connes:unpublished73} showed that $L^\infty[0,1]$ can be the centralizer of an almost periodic state of an arbitrary Krieger factor (actually he proved that this phenomenon characterizes Krieger factors, or equivalently hyperfinite (diffuse) factors due to Connes--Haagerup classification theory for injective or hyperfinite factors). Connes and St{\o}rmer \cite{ConnesStormer:JFA78} showed that any non-type I factor with separable predual always has a faithful normal state whose centralizer is of type II$_1$. Haagerup and St{\o}rmer \cite[Theorem 11.1]{HaagerupStormer:Adv90} strengthened Connes--St{\o}mer's result, especially proved the same result for any $\sigma$-finite von Neumann algebra without type I component. It follows from Ozawa's solidity of $L(\mathbb{F}_\infty)$ \cite{Ozawa:ActaMath04} that a non-injective type III$_1$ factor in the class of free Araki--Woods factors  admits only injective centralizers. (See \cite{Houdayer:JInstMathJussieu10} for related results around this.) However it seems, to the best of our  knowledge, that the question is not yet explicitly answered. Probably many specialists believe that any finite von Neumann algebra can be. We would like to point out the next fact, which in particular shows that $((M_c)_{\varphi|_{M_c}})'\cap M_c^\omega = \mathbb{C}$ may fail to hold in general.   

\begin{proposition}\label{P-2.4} For a given finite von Neumann algebra $N$ with a faithful normal tracial state $\tau$ there is a type III$_1$ factor $M$ with a faithful normal state $\varphi$ such that the centralizer $M_\varphi$ with $\varphi|_{M_\varphi}$ is exactly $N$ with $\tau$. 
\end{proposition}  
\begin{proof} 
Let $R_\infty$ be the unique hyperfinite type III$_1$ factor. It is known (see \cite[\S3]{HermanTakesaki:CMP70} and also \cite[p.246--247]{NeshveyevStormer:EntropyBook}) that there is a faithful normal state $\psi$ on $R_\infty$ such that the modular operator $\Delta_\psi$ has no eigenvalue on the orthogonal complement of the representing vector $\xi_\psi$ of $\psi$ in $L^2(R_\infty,\psi)$. Let $(M,\varphi)$ be the free product of $(N,\tau)$ and $(R_\infty,\psi)$. By \cite[Theorem 3.4]{Ueda:Preprint10} the free product von Neumann algebra $M$ is a factor of type III$_1$. Moreover we can prove, see \cite[Lemma 7]{Barnett:PAMS95}, that the centralizer $M_\varphi$ is exactly $N$ by using the simple fact that any tensor product $U_t\otimes V_t$ of $1$-parameter unitary group $U_t$ without eigenvector and arbitrary one $V_t$ (even possibly to be the trivial one) has no eigenvector (which can easily be seen by using e.g.~\cite[Theorem VI.2.9 in p.138]{Katznelson:Book}).   
\end{proof} 

It is easy to see that the modular operator associated with the free product state $\varphi$ constructed in Proposition \ref{P-2.4} has no eigenvalue except $1$. However an almost periodic state may still exist on $M$, but it is likely that $M$ always has no such state. Hence it is desirable to find a necessary and sufficient condition for the existence of almost periodic states on the diffuse factor part $M_c$ of arbitrary free product von Neumann algebra $M$. This question will be answered in the next section.   

\section{$\tau$-invariant $\tau(M_c)$} 

Throughout this section let us assume that $M_1$ and $M_2$ are von Neumann algebras {\it with separable preduals}, since automorphism analysis will play a key r\^{o}le in this section, and also $\varphi_1$ and $\varphi_2$ are arbitrary faithful normal states on $M_1$ and $M_2$, respectively. Denote by $(M,\varphi)$ the free product of $(M_1,\varphi_1)$ and $(M_2,\varphi_2)$. The main purpose of this section is to compute the $\tau$-invariant $\tau(M_c)$ of the diffuse factor part $M_c$ and to clarify when the diffuse factor part $M_c$ possesses an almost periodic state or weight.  

Let us begin by recalling some definitions. For a given factor $N$ with separable predual, $N$ is said to be full if $\mathrm{Int}(N)$ is closed in $\mathrm{Aut}(N)$ endowed with the so-called $u$-topology, see \cite[\S III]{Connes:JFA74} (or \cite[Ch.XIV,\S3]{Takesaki:Book3}), and the $\tau$-invariant $\tau(N)$ of a full factor $N$ is defined to be the weakest topology on $\mathbb{R}$ that makes the so-called modular homomorphism $t \in \mathbb{R} \mapsto \delta_N(t)\in \mathrm{Out}(N)$ be continuous, see \cite[\S V]{Connes:JFA74}. Here $\mathrm{Out}(N) := \mathrm{Aut}(N)/\mathrm{Int}(N)$ (with the quotient map $\varepsilon_N$) becomes a Polish (= separable metrizable complete) group with the quotient topology induced from the $u$-topology and define $\delta_N(t) := \varepsilon_N(\sigma_t^\psi) \in \mathrm{Out}(N)$ with an arbitrary fixed faithful normal state or semifinite weight $\psi$ on $N$.   

The next proposition is most technically involved in the present notes.  

\begin{proposition}\label{P-3.1} If either $M_1$ or $M_2$ is diffuse, then for any sequence $(t_m)_m$ of real numbers, $\delta_M(t_m) \longrightarrow \varepsilon_M(\mathrm{Id})$ in $\mathrm{Out}(M)$ as $m \rightarrow \infty$ if and only if both $\sigma_{t_m}^{\varphi_i} \longrightarrow \mathrm{Id}$ in $\mathrm{Aut}(M_i)$, $i=1,2$, or equivalently $\sigma_{t_m}^\varphi \longrightarrow \mathrm{Id}$ in $\mathrm{Aut}(M)$, as $m \rightarrow \infty$. 
\end{proposition} 
\begin{proof} We will borrow several facts and arguments from \cite[\S\S2.2 and \S3]{Ueda:Preprint10} in what follows.  

It suffices to show the `only if' part. Take a sequence $(t_m)_m$ of real numbers such that $\delta_M(t_m) \longrightarrow \varepsilon_M(\mathrm{Id})$ in $\mathrm{Out}(M)$ as $m \rightarrow \infty$. Then there is a sequence $(u(m))_m$ in $M^u$ such that $\mathrm{Ad}u(m)\circ\sigma_{t_m}^\varphi \longrightarrow \mathrm{Id}$ in $\mathrm{Aut}(M)$ as $m\rightarrow\infty$. Let us choose and fix an arbitrary free ultrafilter $\omega \in \beta(\mathbb{N})\setminus\mathbb{N}$. 

By symmetry we may and do assume that $M_1$ is diffuse. As in \cite[Theorem 3.4]{Ueda:Preprint10} there is a faithful normal state $\psi$ on $M_1$ so that $(M_1)_\psi$ is diffuse, and thus one can choose two unitaries $a,b \in (M_1)_\psi$ in such a way that $\varphi_1(a^n) = \delta_{n0} = \psi(b^n)$, see e.g.~the proof of \cite[Theorem 3.7]{Ueda:Preprint10}. Denote by $E_1 : M \rightarrow M_1$ the $\varphi$-preserving conditional expectation, see \cite[Lemma 2.1]{Ueda:Preprint10}. 

Since $\mathrm{Ad}u(m)\circ\sigma_{t_m}^\varphi \longrightarrow \mathrm{Id}$ in $\mathrm{Aut}(M)$ as $m\rightarrow\infty$, one has $\Vert \varphi - \varphi\circ\mathrm{Ad}u(m)\Vert_{M_*} = \Vert \varphi\circ\mathrm{Ad}u(m)^* - \varphi\Vert_{M_*} = \Vert \varphi\circ\sigma_{-t_m}^\varphi\circ\mathrm{Ad}u(m)^* - \varphi\Vert_{M_*} = \Vert \varphi - \varphi\circ\mathrm{Ad}u(m)\circ\sigma_{t_m}^\varphi\Vert_{M_*} \longrightarrow 0$ as $m \rightarrow \infty$. Hence, for any bounded sequence $(x(m))_m$ of $M$ with $x(m) \longrightarrow 0$ in $\sigma$-strong* topology as $m \rightarrow \omega$ we have 
\begin{align*} 
\Vert x(m)^* u(m)^*\Vert^2_\varphi 
&= 
\varphi(u(m) x(m) x(m)^* u(m)^*) \\
&\leq 
|(\varphi\circ\mathrm{Ad}u(m) - \varphi)(x(m) x(m)^*)| + \varphi(x(m) x(m)^*) \\
&\leq 
\sup_m\Vert x(m)\Vert_\infty^2 \Vert\varphi\circ\mathrm{Ad}u(m) - \varphi\Vert_{M_*} + \Vert x(m)^*\Vert_\varphi^2 \underset{m\rightarrow\omega}{\longrightarrow} 0, \\
\Vert x(m) u(m) \Vert_\varphi^2 &= \psi(u(m)^* x(m)^* x(m) u(m)) \\
&\leq |(\varphi\circ\mathrm{Ad}u(m)^* - \varphi)(x(m)^* x(m))| + \varphi(x(m)^* x(m)) \\
&\leq 
\sup_m\Vert x(m)\Vert_\infty^2 \Vert\varphi\circ\mathrm{Ad}u(m)^* - \varphi\Vert_{M_*} + \Vert x(m)\Vert_\varphi^2 \underset{m\rightarrow\omega}{\longrightarrow} 0 
\end{align*} 
so that $(u(m))_m$ represents a unitary $u$ in the ultraproduct $M^\omega$. Note that 
$$
\mathrm{Ad}(u(m) [D\varphi_1:D\psi]_{t_m})\circ\sigma_{t_m}^{\psi\circ E_1} = \mathrm{Ad}u(m)\circ\sigma_{t_m}^\varphi \longrightarrow \mathrm{Id}
$$ 
in $\mathrm{Aut}(M)$ as $m\rightarrow\infty$. Write $v(m) := [D\varphi_1:D\psi]_{t_m} \in M_1^u$ for simplicity, and set $w(m) := u(m) v(m)$. We apply the same argument as above to the pair $w(m)$ and $\psi\circ E_1$, and consequently we see that $(w(m))_m$ represents a unitary $w \in M^\omega$. We have $v := u^* w = [(u(m)^* w(m))_m] = [(v(m))_m]$ so that $(v(m))_m$ also represents a unitary $v \in M_1^\omega$. Let $y \in M^\omega$ be an arbitrary element with representative $(y(m))_m$, and $\chi$ be an arbitrary faithful normal state on $M$. For any bounded sequence $(x(m))_m$ of $M$ with $x(m) \longrightarrow 0$ in $\sigma$-strong* topology as $m \rightarrow \omega$ one has $\Vert\sigma_{\mp t_m}^\chi(x(m))\Vert_\chi = \Vert x(m)\Vert_\chi \longrightarrow 0$ and $\Vert\sigma_{\mp t_m}^\chi(x(m)^*)\Vert_\chi = \Vert x(m)^*\Vert_\chi \longrightarrow 0$ as $m\rightarrow\omega$ so that $\sigma_{\mp t_m}^\chi(x(m)) \longrightarrow 0$ in $\sigma$-strong* topology as $m \rightarrow \omega$. Therefore we get $\Vert x(m)\sigma_{\pm t_m}^\chi(y(m))\Vert_\chi = 
\Vert \sigma_{\mp t_m}^\psi(x(m)) y(m) \Vert_\chi \longrightarrow 0$ and $\Vert x(m)^*\sigma_{\pm t_m}^\chi(y(m)^*)\Vert_\chi = 
\Vert \sigma_{\mp t_m}^\psi(x(m)^*) y(m)^* \Vert_\chi \longrightarrow 0$ as $m\rightarrow\omega$. These establish that $(\sigma^\chi_{\pm t_m}(y(m)))_m$ represents an element in $M^\omega$, which we denote by $y_{(\pm)}^\chi$ in what follows. (This indeed shows that $(\sigma_{\pm t_m}^\chi)_m$ is `semiliftable' in the sense of Ocneanu in \cite[\S5.2]{Ocneanu:LNM1138}.) 

Let us prove that $w = E_1^\omega(w) \in M_1^\omega$. For any $x \in (M_1)_\psi$ one has 
$$
wxw^* = [(w(m) x w(m)^*)_m] = [(w(m) \sigma_{t_m}^{\psi\circ E_1}(x) w(m)^*)_n] = [(x)_m] = x
$$ 
inside $M^\omega$, since $\mathrm{Ad}w(m)\circ\sigma_{t_m}^{\psi\circ E_1} \longrightarrow \mathrm{Id}$ in $\mathrm{Aut}(M)$ as $m\rightarrow\infty$ (and thus the same holds true when $m \rightarrow \omega$). Hence we get $w \in ((M_1)_\psi)' \cap M^\omega \subseteq \{a,b\}'\cap M^\omega$. One can choose an invertible $y\in M_2^\circ$ ({\it n.b.}~$M_2 \neq \mathbb{C}$), see the proof of \cite[Theorem 3.7]{Ueda:Preprint10}. As seen in the previous paragraph  the sequence $(\sigma_{t_m}^{\psi\circ E_1}(y)))_m$ gives $y^{\psi\circ E_1}_{(+)} \in M^\omega$. Then we compute $wy_{(+)}^{\psi\circ E_1}w^* = [(w(m) \sigma_{t_m}^{\psi\circ E_1}(y)w(m)^*)_m] = [(y)_m] = y$ inside $M^\omega$ thanks to $\mathrm{Ad}w(m)\circ\sigma_{t_m}^{\psi\circ E_1} \longrightarrow \mathrm{Id}$ in $\mathrm{Aut}(M)$ as $m\rightarrow\infty$ again. Hence $y(w-E_1^\omega(w)) + (yE_1^\omega(w)-E_1^\omega(w)y_{(+)}^{\psi\circ E_1}) + (E_1^\omega(w)-w)y_{(+)}^{\psi\circ E_1} =yw - wy_{(+)}^{\psi\circ E_1} = 0$ inside $M^\omega$. For the purpose here we will prove, by the same technique as in \cite[Proposition 3.5]{Ueda:Preprint10}, that $y(w-E_1^\omega(w))$ is orthogonal to the others in $L^2(M^\omega,(\psi\circ E_1)^\omega)$ in what follows. As in \cite[Proposition 3.5]{Ueda:Preprint10} we write $M_1^\triangledown := \mathrm{Ker}(\psi)$, and denote by $P_1, P_2, P_3, P_4$ the projections from $\mathcal{H} := L^2(M,\psi\circ E_1)$ onto the closed subspaces spanned (via $\Lambda_{\psi\circ E_1}$) by the following sets of words 
$$
M_1^\circ M_2^\circ \cdots M_1^\triangledown, \quad 
M_1^\circ \cdots M_2^\circ, \quad 
M_2^\circ \cdots M_1^\triangledown, \quad 
M_2^\circ \cdots M_2^\circ,  
$$
respectively. We also denote by $\bar{E}_1$ the projection from $\mathcal{H}$ onto the closure of $M_1$ (as a subspace of $\mathcal{H}$ via $\Lambda_{\psi\circ E_1}$) induced by $E_1$, see the proof of \cite[Lemma 2.1]{Ueda:Preprint10}. Remark (see the proof of \cite[Proposition 3.5]{Ueda:Preprint10}) that $I_\mathcal{H} = \bar{E}_1 + P_1 + P_2 + P_3 + P_4$ (and $\bar{E}_1, P_1, P_2, P_3, P_4$ are mutually orthogonal). With replacing $u,v$ there by the above unitaries $a,b$ the exactly same argument as in the proof of \cite[Proposition 3.5]{Ueda:Preprint10} shows that   
for each $\delta>0$ there is a neighborhood $W_\delta$ in $\beta(\mathbb{N})$ at $\omega$ such that 
\begin{equation}\label{Eq-3.1} 
\Vert(P_2+P_3+P_4)\Lambda_{\psi\circ E_1}(w(m))\Vert_{\psi\circ E_1} < \delta
\end{equation} 
as long as $m \in W_\delta\cap\mathbb{N}$. We then regard $L^2(M^\omega,(\psi\circ E_1)^\omega)$ as a closed subspace of the ultraproduct $\mathcal{H}^\omega$. One can see, by using \eqref{Eq-3.1}, that 
$$
\Lambda_{(\psi\circ E_1)^\omega}(y(w-E_1^\omega(w))) = \big[\big(yP_1\Lambda_{\psi\circ E_1}(w(m))\big)_m\big]
$$
in $\mathcal{H}^\omega$ with $\mathcal{H} = L^2(M,\psi\circ E_1)$. See the estimate \eqref{Eq-3.2} below or the proof of \cite[Proposition 3.5]{Ueda:Preprint10} for its detailed derivation. Also it is trivial that 
$$
\Lambda_{(\psi\circ E_1)^\omega}(yE_1^\omega(w)-E_1^\omega(w)y_{(+)}^{\psi\circ E_1}) = 
\big[\big(\Lambda_{\psi\circ E_1}(y E_1(w(m)) - E_1(w(m))\sigma_{t_m}^{\psi\circ E_1}(y))\big)_m\big]
$$
in $\mathcal{H}^\omega$. Consider a smoothing element $y_n$ ($n\in\mathbb{N}$) of $y$ with respect to $\sigma^{\psi\circ E_1}$ defined to be  
$$
y_n := \frac{1}{\sqrt{n\pi}}\int_{-\infty}^\infty e^{-t^2/n}\sigma_t^{\psi\circ E_1}(y)\,dt,  
$$ 
which falls into the $\sigma$-strong closure of the linear span of $M_1 M_2^\circ M_1$ and converges to $y$ in $\sigma$-weak topology as $n \rightarrow \infty$ (see e.g.~the proof of \cite[Proposition 3.5]{Ueda:Preprint10}). Note that $\sigma_{-i/2}^{\psi\circ E_1}(\sigma_{t_m}^{\psi\circ E_1}(y_n)) = \sigma_{t_m}^{\psi\circ E_1}(\sigma_{-i/2}^{\psi\circ E_1}(y_n))$, and thus for each $n$ we have, by \eqref{Eq-3.1},  
\begin{align}\label{Eq-3.2} 
&\big\Vert\Lambda_{(\psi\circ E_1)^\omega}((w-E_1^\omega(w))y_{n\,(+)}^{\psi\circ E_1}) \notag\\
&\phantom{aaaaaaaaaaaaa}-
\big[\big(J\sigma_{t_m}^{\psi\circ E_1}(\sigma_{-i/2}^{\psi\circ E_1}(y_n))^* J P_1\Lambda_{\psi\circ E_1}(w(m))\big)_m\big]\big\Vert_{(\psi\circ E_1)^\omega} \notag\\
&= \lim_{m\rightarrow\omega} 
\big\Vert J\sigma_{t_m}^{\psi\circ E_1}(\sigma_{-i/2}^{\psi\circ E_1}(y_n))^* J (\Lambda_{\psi\circ E_1}(w(m)-E_1(w(m))) \notag\\
&\phantom{aaaaaaaaaaaaaaaaaaaaaaaaaaaaaaaaaaaaaaaaa}- P_1\Lambda_{\psi\circ E_1}(w(m)))\big\Vert_{\psi\circ E_1} \\
&\leq \Vert \sigma_{-i/2}^{\psi\circ E_1}(y_n)\Vert_\infty \sup_{m \in W_\delta\cap\mathbb{N}}\Vert(P_2+P_3+P_4)\Lambda_{\psi\circ E_1}(w(m))\Vert_{\psi\circ E_1} \notag\\ 
&< \Vert \sigma_{-i/2}^{\psi\circ E_1}(y_n)\Vert_\infty \delta, \notag
\end{align}   
where $J$ is the modular conjugation of $M \curvearrowright \mathcal{H} = L^2(M,\psi\circ E_1)$. Since $\delta>0$ is arbitrary, we get 
$$
\Lambda_{(\psi\circ E_1)^\omega}((w-E_1^\omega(w))y_{n\,(+)}^{\psi\circ E_1}) = 
\big[\big(J\sigma_{-i/2}^{\psi\circ E_1}(\sigma_{t_m}^{\psi\circ E_1}(y_n))^*J P_1\Lambda_{\psi\circ E_1}(w(m))\big)_m\big]
$$
in $\mathcal{H}^\omega$ for each $n$. Since $\sigma_{t_m}^{\psi\circ E_1}(y_n) = v(m)^* \sigma_{t_m}^\varphi(y_n) v(m)$ still falls in the $\sigma$-strong closure of $M_1 M_2^\circ M_1$ ({\it n.b.}~$v(m) = [D\varphi_1:D\psi]_{t_m} \in M_1$), the same argument as in the proof of \cite[Proposition 3.5]{Ueda:Preprint10} shows that $\Lambda_{(\psi\circ E_1)^\omega}(y(w-E_1^\omega(w)))$ is orthogonal to $\Lambda_{(\psi\circ E_1)^\omega}(yE_1^\omega(w)-E_1^\omega(w)y_{(+)}^{\psi\circ E_1})$ and  also to all $\Lambda_{(\psi\circ E_1)^\omega}((w-E_1^\omega(w))y_{n\,(+)}^{\psi\circ E_1})$'s. Notice here that $(\sigma_{-t_m}^{\psi\circ E_1}(w(m)))_m,\, (\sigma_{-t_m}^{\psi\circ E_1}(y))_m$ represent $w_{(-)}^{\psi\circ E_1},\, y_{(-)}^{\psi\circ E_1} \in M^\omega$, respectively, as seen before. We have, for $z = y$ or $y_n$,  
\begin{align*} 
&(\psi\circ E_1)^\omega((w-E_1^\omega(w))^* y^* (w-E_1^\omega(w))z_{(+)}^{\psi\circ E_1}) \\
&= 
\lim_{m\rightarrow\omega} \psi\circ E_1((w(m)^* - E_1(w(m)^*))y^* (w(m)-E_1(w(m))) \sigma_{t_m}^{\psi\circ E_1}(z)) \\
&= 
\lim_{m\rightarrow\omega} \psi\circ E_1((\sigma_{-t_m}^{\psi\circ E_1}(w(m))^* - E_1(\sigma_{-t_m}^{\psi\circ E_1}(w(m))^*)) \\
&\phantom{aaaaaaaaaaaaaa}\ \sigma_{-t_m}^{\psi\circ E_1}(y^*)(\sigma_{-t_m}^{\psi\circ E_1}(w(m))-E_1(\sigma_{-t_m}^{\psi\circ E_1}(w(m))))z) \\
&= 
(\psi\circ E_1)^\omega((w_{(-)}^{\psi\circ E_1} - E_1^\omega(w_{(-)}^{\psi\circ E_1}))^* y_{(-)}^{\psi\circ E_1}{}^*(w_{(-)}^{\psi\circ E_1} - E_1^\omega(w_{(-)}^{\psi\circ E_1}))z), 
\end{align*}
and hence 
\begin{align*} 
&(\Lambda_{(\psi\circ E_1)^\omega}((w-E_1^\omega(w))y_{(+)}^{\psi\circ E_1}|\Lambda_{(\psi\circ E_1)^\omega}(y(w-E_1^\omega(w))))_{(\psi\circ E_1)^\omega} \\
&= (\psi\circ E_1)^\omega((w_{(-)}^{\psi\circ E_1} - E_1^\omega(w_{(-)}^{\psi\circ E_1}))^* y_{(-)}^{\psi\circ E_1}{}^*(w_{(-)}^{\psi\circ E_1} - E_1^\omega(w_{(-)}^{\psi\circ E_1}))y) \\
&= 
\lim_{n\rightarrow\infty}(\psi\circ E_1)^\omega((w_{(-)}^{\psi\circ E_1} - E_1^\omega(w_{(-)}^{\psi\circ E_1}))^* y_{(-)}^{\psi\circ E_1}{}^*(w_{(-)}^{\psi\circ E_1} - E_1^\omega(w_{(-)}^{\psi\circ E_1}))y_n) \\
&= 
\lim_{n\rightarrow\infty}(\Lambda_{(\psi\circ E_1)^\omega}((w-E_1^\omega(w))y_{n\,(+)}^{\psi\circ E_1}|\Lambda_{(\psi\circ E_1)^\omega}(y(w-E_1^\omega(w))))_{(\psi\circ E_1)^\omega} = 0. 
\end{align*} 
Consequently $y(w-E_1^\omega(w))$ is orthogonal to $(w-E_1^\omega(w))y_{(+)}^{\psi\circ E_1}$ too. Therefore we get $\Vert y(w-E_1^\omega(w))\Vert_{(\psi\circ E_1)^\omega} \leq \Vert yw - wy_{(+)}^{\psi\circ E_1}\Vert_{(\psi\circ E_1)^\omega} = 0$, implying $w = E_1^\omega(w) \in M_1^\omega$ since $y$ is invertible. 

Since $w \in M_1^\omega$ we have $u = wv^* \in M_1^\omega$. For the above $y \in M_2^\circ$ we have $u y_{(+)}^\varphi u^* = [(\mathrm{Ad}u(m)\circ\sigma_{t_m}^\varphi(y))_m] = [(y)_m] = y$, since $\mathrm{Ad}u(m)\circ\sigma_{t_m}^\varphi \longrightarrow \mathrm{Id}$ in $\mathrm{Aut}(M)$ $m \rightarrow \infty$. Thus $(u-\varphi^\omega(u)1)y_{(+)}^\varphi + \varphi^\omega(u)y_{(+)}^\varphi = uy_{(+)}^\varphi = yu = \varphi^\omega(u)y + y(u-\varphi^\omega(u)1)$. Since $M_1^\omega$ and $M_2^\omega$ are free in $(M^\omega,\varphi^\omega)$ (see e.g.~\cite[Proposition 4]{Ueda:TAMS03}), we have $\Vert y\Vert_{\varphi^\omega}\Vert u-\varphi^\omega(u)1\Vert_{\varphi^\omega} = \Vert y(u-\varphi^\omega(u)1)\Vert_{\varphi^\omega} = 0$ so that $\lim_{m\rightarrow\omega}\Vert u(m) - \varphi(u(m))1\Vert_\varphi = \Vert u - \varphi^\omega(u)1\Vert_{\varphi^\omega} = 0$ thanks to $y \neq 0$. Since our choice of $\omega \in \beta(\mathbb{N})\setminus\mathbb{N}$ is arbitrary, we get $\lim_{m\rightarrow\infty}\Vert u(m) - \varphi(u(m))1\Vert_\varphi = 0$. Hence we conclude that $\sigma_{t_m}^\varphi \longrightarrow \mathrm{Id}$ in $\mathrm{Aut}(M)$ as $m \rightarrow \infty$ in the exactly same way as in the proof of \cite[Theorem 5.2]{Connes:JFA74}. 
\end{proof} 

Here is the main theorem of the present notes. 

\begin{theorem}\label{T-3.2} Assume that $M_1 \neq \mathbb{C} \neq M_2$ and $\mathrm{dim}(M_1)+\mathrm{dim}(M_2) \geq 5$. Then the $\tau$-invariant $\tau(M_c)$ of the diffuse factor part $M_c$ of the free product von Neumann algebra $M$ is the weakest topology on $\mathbb{R}$ that makes the both mappings $t \in \mathbb{R} \mapsto \sigma_t^{\varphi_i} \in \mathrm{Aut}(M_i)$, $i = 1,2$, be continuous. 
\end{theorem}
\begin{proof} We may assume that $M_c$ is of type III, that is, either $\varphi_1$ or $\varphi_2$ is non-tracial. Let us decompose $M_i = M_{id}\oplus M_{ic}$ into the `type I with discrete center' part and the diffuse part, $i = 1,2$. We may and do further assume that either $M_{1c} \neq 0$ or $M_{2c} \neq 0$. (Otherwise $((M_c)_{\varphi|_{M_c}})' \cap M_c^\omega = \mathbb{C}$ by Theorem \ref{T-2.2} (or \cite[Remark 4.2 (4)]{Ueda:Preprint10}), which immediately implies the desired assertion, see the proof of Proposition \ref{P-3.1}. In fact, if $\mathrm{Ad}u(m)\circ\sigma_{t_m}^\varphi \longrightarrow \varepsilon_{M_c}(\mathrm{Id})$ in $\mathrm{Aut}(M_c)$, then $(u(m))_m$ represents a unitary in $((M_c)_{\varphi|_{M_c}})'\cap M_c^\omega=\mathbb{C}$, implying $\lim_{m\rightarrow\omega}\Vert u(m)-(1/\varphi(1_{M_c}))\varphi(u(m))1_{M_c}\Vert_\varphi = 0$.) In what follows we assume that $M_{1c} \neq 0$. By \cite[Theorem 4.1]{Ueda:Preprint10} the compressed algebra $1_{M_{1c}}M1_{M_{1c}}$ is isomorphic to $M_c$ since $M_c$ is assumed to be of type III. Thus $\tau(M_c) = \tau(1_{M_{1c}}M1_{M_{1c}})$ holds, and write $p:=1_{M_{1c}}$ for simplicity. By Dykema's free etymology technique (see e.g.~\cite[Lemma 2.2]{Ueda:Preprint10}) one has  
$$
(pMp,\,(1/\varphi_1(p))\varphi|_{pMp}) = 
(M_{1c},(1/\varphi_1(p))\varphi_1|_{M_{1c}})\star (pNp,(1/\varphi_1(p))\varphi|_{1_{pNp}}), 
$$ 
where $N = (M_{1d}\oplus\mathbb{C}p)\vee M_2$. Suppose that $t_m \rightarrow 0$ in $\tau(M_c) = \tau(pMp)$. By Proposition \ref{P-3.1} $\sigma_{t_m}^{\varphi_1|_{M_{1c}}} \longrightarrow \mathrm{Id}$ in $\mathrm{Aut}(M_{1c})$ and $\sigma_{t_m}^{\varphi|_{pNp}} \longrightarrow \mathrm{Id}$ in $\mathrm{Aut}(pNp)$. Note that $N$ is the free product von Neumann algebra of $M_{1d}\oplus\mathbb{C}p$ and $M_2$ with respect to $\varphi_1|_{M_{1d}\oplus\mathbb{C}p}$ and $\varphi_2$. Any eigenvector $x \in M_{2d}$ of the modular action $\sigma^{\varphi_2|_{M_{2d}}} = \sigma^{\varphi_2}|_{M_{2d}}$, i.e., $\sigma_t^{\varphi_2}(x) = \lambda^{it}x$ for some $\lambda>0$, satisfies $pxp \neq 0$. It follows that $\sigma_{t_m}^{\varphi_2|_{M_{2d}}} \longrightarrow \mathrm{Id}$ in $\mathrm{Aut}(M_{2d})$ since the restrictions $\varphi_2|_{M_{2d}}$ are almost periodic. When $M_{2c} = 0$, i.e., $M_2 = M_{2d} \neq \mathbb{C}$, one can choose an eigenvector $y \in M_{2d}^\circ$ of $\sigma^{\varphi_2|_{M_{2d}}}=\sigma^{\varphi_2}|_{M_{2d}}$, and then any non-trivial eigenvector $x \in M_{1d}$ of $\sigma^{\varphi_1|_{M_{1d}}}=\sigma^{\varphi_1}|_{M_{1d}}$ satisfies $pyxy^* p \neq 0$, and thus $\sigma_{t_m}^{\varphi_1|_{M_{1d}}} \longrightarrow \mathrm{Id}$ in $\mathrm{Aut}(M_{1d})$ as above. When $M_{2c} \neq 0$, we also have, by symmetry, $\sigma_{t_m}^{\varphi_2|_{M_{2c}}}\longrightarrow \mathrm{Id}$ in $\mathrm{Aut}(M_{2c})$ and $\sigma_{t_m}^{\varphi_1|_{M_{1d}}} \longrightarrow \mathrm{Id}$ in $\mathrm{Aut}(M_{1d})$. Consequently the desired assertion follows. \end{proof}  

\begin{corollary}\label{C-3.3} Assume that  $M_1 \neq \mathbb{C} \neq M_2$ and $\mathrm{dim}(M_1)+\mathrm{dim}(M_2) \geq 5$. Then a necessary and sufficient condition for the existence of almost periodic state or weight on the diffuse factor part $M_c$ of the free product von Neumann algebra $M$ is that both the given $\varphi_1$ and $\varphi_2$ are almost periodic. 
\end{corollary}
\begin{proof} We may and do assume that $M_c$ is of type III$_1$. Suppose that $M_c$ has an almost periodic state. By \cite[Theorem 4.7]{Connes:JFA74} there is an almost periodic weight $\psi$ on $M_c$ such that the point spectrum of $\Delta_\psi$ is exactly $\mathrm{Sd} := \mathrm{Sd}(M_c)$. Then by \cite[Proposition 1.1]{Connes:JFA74} $t \in \mathbb{R} \mapsto \sigma_t^\psi$ can continuously be extended to the dual group $\widehat{\mathrm{Sd}}$ ($\mathrm{Sd}$ is equipped with its discrete topology), where $\mathbb{R}$ is continuously, faithfully (\cite[Corollary 4.11]{Connes:JFA74}) embedded into $\widehat{\mathrm{Sd}}$ whose range is dense. Note that $\sigma_{t_m}^\psi \rightarrow \mathrm{Id}$ in $\mathrm{Aut}(M_c)$ implies $\delta_{M_c}(t_m) \longrightarrow \varepsilon_{M_c}(\mathrm{Id})$ in $\mathrm{Out}(M_c)$, and hence by Theorem \ref{T-3.2} $\sigma_{t_m}^{\varphi_i} \rightarrow \mathrm{Id}$ in $\mathrm{Aut}(M_i)$, $i=1,2$. Thus $t \in \mathbb{R} \mapsto \sigma_t^{\varphi_i} \in \mathrm{Aut}(M_i)$ can continuously be extended to the whole $\widehat{\mathrm{Sd}}$, $i=1,2$.  
Hence both $\varphi_i$, $i=1,2$, must be almost periodic by \cite[Proposition 1.1]{Connes:JFA74}.  
\end{proof} 

\section{Concluding Remarks} 

The previous paper \cite{Ueda:Preprint10} and the present notes solve the questions of 
\begin{itemize} 
\item its factoriality (\cite[Theorem 4.1]{Ueda:Preprint10}), 
\item determining its Murray--von Neumann--Connes type (\cite[Theorem 4.1]{Ueda:Preprint10}), 
\item its fullness (\cite[Theorem 4.1]{Ueda:Preprint10}),
\item computing its $\mathrm{Sd}$-invariant (Corollary \ref{C-2.3} of the present notes), 
\item computing its $\tau$-invariant (Theorem \ref{T-3.2} of the present notes) 
\end{itemize}    
for arbitrary free product von Neumann algebra. Those results in particular show that the resulting free product von Neumann algebra certainly `remembers' the given states, that is, the free product state is `special' in some sense. 

One more algebraic invariant related to full type III$_1$ factors was introduced by Shlyakhtenko \cite{Shlyakhtenko:TAMS05}. However we cannot yet deal with it. Also Connes' bicentralizer problem should be examined for free product von Neumann algebras. In fact, Houdayer \cite{Houdayer:PAMS09} showed that any type III$_1$ free Araki--Woods factor has the trivial bicentralizer. In the direction we can  confirm, by \cite[Corollary 3.2, Theorem 4.1]{Ueda:Preprint10} together with Haagerup's solution \cite{Haagerup:ActaMath86}, that the bicentralizer problem is affirmative for any type III$_1$ factor arising as (a direct summand of) free product of hyperfinite von Neumann algebras.  However we do not know whether or not the problem is affirmative in general. 

\section*{Acknowledgment} We thank the referee for pointing out a typo and giving a comment related to Proposition \ref{P-2.4}.
}

\end{document}